\documentclass{amsart}
\usepackage{graphicx}
\usepackage{amsmath}
\usepackage{amsfonts}
\usepackage{amssymb}
\usepackage[T1]{fontenc}
\newtheorem{theorem}{Theorem}

\newtheorem{corollary}[theorem]{Corollary}

\newtheorem{definition}[theorem]{Definition}
\newtheorem{example}[theorem]{Example}

\newtheorem{lemma}[theorem]{Lemma}

\newtheorem{proposition}[theorem]{Proposition}
\newtheorem{remark}[theorem]{Remark}

\begin{document}

\title[The structure of completely meet irreducible congruences...]{The structure of completely meet irreducible congruences in strongly Fregean algebras}
\author{Katarzyna S\l omczy\'{n}ska}
\address{Institute of Mathematics, Pedagogical University \\ Podchor\k{a}\.{z}ych 2, 30-084 Krak\'{o}w}
\email{kslomcz@up.krakow.pl}

\begin{abstract}
{A strongly Fregean algebra is an algebra such that the class of its homomorphic images is Fregean and the variety generated by this algebra is congruence modular. To understand the structure of these algebras we study the prime intervals projectivity relation in the posets of their completely meet irreducible congruences and show that its cosets have natural structure of Boolean group. In particular, this approach allows us to represent congruences and elements of such algebras as the subsets of upward closed subsets of these posets with some special properties.}
\end{abstract}

\maketitle

\section{Introduction}

The word Fregean comes from Frege's idea that sentences should denote their
logical values. This idea was an inspiration for Pigozzi \cite{Pig91}, who
transferred it to the field of universal algebra, see \cite{IdzSloWro97},
\cite{Fonetal03}, \cite{CzePig04} ,\cite{CzePig04A} and \cite{IdzSloWro09}%
\ for thorough historical discussion. A \textsl{Fregean algebra} is as an
algebra $\mathbf{A}$ with a distinguished constant term $1$ satisfying two
axioms: $\Theta_{\mathbf{A}}\left(  1,a\right)  =\Theta_{\mathbf{A}}\left(
1,b\right)  $ implies $a=b$ for $a,b\in A$ (congruence orderability), and
$1/\alpha=1/\beta$ implies $\alpha=\beta$ for $\alpha,\beta\in\mathsf{Con}%
\left(  \mathbf{A}\right)  $ ($1$-regularity). A class of algebras of the same
type is called Fregean if all algebras in this class are Fregean with respect
to the common constant term $1$. Many natural examples of \textsl{Fregean
varieties}, like Boolean algebras, Heyting algebras, Brouwerian semilattices,
Boolean groups, equivalential algebras (i.e., the equivalential subreducts of
Heyting algebras), equivalential algebras with negation, Hilbert algebras or
Hilbert algebras with supremum, come from the algebraization of fragments of
classical or intuitionistic logics. Fregean varieties are congruence modular,
but not necessarily congruence permutable: Hilbert algebras can serve as an
example here. To see that this distinction is important, note that the
structure of congruence permutable Fregean varieties is quite well understood,
in particular, it was proved in \cite{IdzSloWro09} that every congruence
permutable Fregean variety consists of algebras that are expansions of
equivalential algebras. On the other hand, the structure of Fregean algebras
in general remains elusive.

The main technical tool used in analyzing the structure of Fregean algebras
was \textsl{prime intervals projectivity} (\textsl{PIP}) relation. (Note,
however, that in the congruence distributive case this relation is trivial.)
In \cite{Slo96} and \cite{IdzSloWro09} we used this relation in the set of
\textsl{join} irreducible elements of congruence lattices of finite algebras
from congruence permutable Fregean varieties to understand their structure.
However, in a series of papers \cite{Slo05}, \cite{Slo08} and \cite{Slo17} we
apply the same relation, but this time in the set of completely \textsl{meet}
irreducible elements of congruence lattice, in particular, to construct free
algebras in the varieties under consideration. This perspective is, in a
sense, dual to the former, though substantially different, since the posets of
join and meet irreducible elements of the congruence lattice need not to be
isomorphic, and the latter approach can be applied also in the infinite case,
in contrast to the former one.

In the present paper we show that one can apply this technique to a broad
class of Fregean algebras also out of the congruence permutability realm. This
is surprising, since we do not know almost nothing about the language of
algebra in this situation, in contrast to the congruence permutable case, when
we are at least sure that it contains equivalence operation. Namely, we shall
consider \textsl{strongly Fregean algebras}, i.e., the algebras such that the
classes of their homomorphic images are Fregean and the varieties generated by
these algebras are congruence modular. This class lies in between Fregean
algebras and Fregean varieties, as every algebra from a Fregean variety must
be strongly Fregean. We know from \cite{IdzSloWro09}\ that strongly Fregean
algebras fulfill the (SC1) condition isolated in \cite{IdzSlo01}. This
condition simplifies studying the PIP relation in the poset of completely meet
irreducible congruences significantly, in particular, it allows us to
introduce the structure of Boolean group into the PIP equivalence classes
(Theorem \ref{Boolean}) and to comprehend their location in the congruence
lattice (Theorem \ref{location}). Moreover, for finite strongly Fregean
algebras we get a formula for the length of the congruence lattice as the sum
of dimensions of these classes treated as vector spaces over the field
$\mathbb{Z}_{2}$ (Theorem~\ref{dimension}).

In the set of upward closed subsets of the poset of completely meet
irreducible congruences of a strongly Fregean algebra $\mathbf{A}$ we can
distinguish two subsets $\mathcal{H}\left(  \mathbf{A}\right)  $ and
$\mathcal{S(}\mathbf{A})$. The larger subset $\mathcal{S(}\mathbf{A})$
consists of sets that intersected with any PIP equivalence class form its
subgroup, whereas the smaller subset $\mathcal{H}\left(  \mathbf{A}\right)  $
comprises the sets $Z$ such that for each PIP equivalence class such that all
the elements larger than the elements of this class are contained in $Z$, the
intersection of $Z$ with this class is either a maximal subgroup of this class
or is equal to this class. The natural map that sends a congruence into the
set of all larger completely meet irreducible congruences gives us an
embedding of $\mathsf{Con}\left(  \mathbf{A}\right)  $ into $\mathcal{S(}%
\mathbf{A})$ that is surjective for finite algebras (Theorem \ref{S(A)}). The
same map sends principle congruences into $\mathcal{H}\left(  \mathbf{A}%
\right)  $ (Proposition \ref{hyperplane}), closed under the natural
equivalence operation (Theorem \ref{H(A)eq}). Much more can be said if
$\mathbf{A}$ is additionally endowed with a principle congruence term (or,
which is here the same, a Malcev term) and, in consequence, its reduct is an
equivalential algebra. Then the natural embedding preserves equivalence
operation (Proposition \ref{H(A)}). Moreover, we show that a finite congruence
orderable algebra with a Malcev term generates a Fregean variety, and in this
case there is a one-to-one correspondence between $A$ and $\mathcal{H}\left(
\mathbf{A}\right)  $ preserving equivalence operation (Theorem \ref{main}).

\section{Prime Intervals Projectivity (PIP) Relation}

Let $\mathbf{A}$ be an algebra such that its congruence lattice $\mathsf{Con}%
\left(  \mathbf{A}\right)  $ is modular. We denote by $\mathsf{Cm}\left(
\mathbf{A}\right)  $ the set of all completely meet irreducible congruences of
$\mathbf{A}$. For each $\eta$ of $\mathsf{Cm}\left(  \mathbf{A}\right)  $
there is a unique congruence $\eta^{+}\in\mathsf{Con}\left(  \mathbf{A}%
\right)  $ such that $\alpha\geq\eta^{+}$ whenever $\alpha>\eta$ for
$\alpha\in\mathsf{Con}\left(  \mathbf{A}\right)  $. In \cite{IdzSlo01} we
considered an equivalence relation \textsl{PIP} (\textsl{prime intervals
projectivity}, denoted by $\sim$) on $\mathsf{Cm}\left(  \mathbf{A}\right)  $
by putting for \linebreak $\varphi,\psi\in\mathsf{Cm}\left(  \mathbf{A}\right)  $:%
\begin{align*}
\varphi &  \sim\psi  \thinspace \thinspace \thinspace   \text{     if and only if     } \thinspace
\text{the prime intervals }I\left[  \varphi,\varphi^{+}\right]  \text{ and
}I\left[  \psi,\psi^{+}\right]  \text{ are projective,}%
\end{align*}
which has been studied for many years in the theory of modular lattices
\cite{Aic18}.

Observe that whenever $I[\alpha_{1},\beta_{1}]\nearrow I[\alpha_{2},\beta
_{2}]\searrow I[\alpha_{3},\beta_{3}]$ holds for three prime intervals in
$\mathsf{Con}(\mathbf{A})$, then there is $\mu\in\mathsf{Cm}(\mathbf{A})$ with
$I[\alpha_{1},\beta_{1}]\nearrow I[\mu,\mu^{+}]\searrow I[\alpha_{3},\beta
_{3}].$ Indeed, it is enough to pick a completely meet irreducible congruence
$\mu\geq\alpha_{2}$ with $\mu\ngeqslant\beta_{2}$. Hence for $\varphi,\psi
\in\mathsf{Cm}(\mathbf{A})$, $\varphi\sim\psi$ there exists $\alpha_{1}%
,\ldots,\alpha_{n},\beta_{1},\ldots,\beta_{n}\in\mathsf{Con}(\mathbf{A})$ and
$\mu_{1},\ldots,\mu_{n-1}\in\mathsf{Cm}(\mathbf{A})$ such that%
\begin{align}
I[\varphi,\varphi^{+}]  &  \searrow I[\alpha_{1},\beta_{1}]\nearrow I[\mu
_{1},\mu_{1}^{+}]\searrow\ldots\nonumber\\
&  \nearrow I[\mu_{n-1},\mu_{n-1}^{+}]\searrow I[\alpha_{n},\beta_{n}]\nearrow
I[\psi,\psi^{+}]\text{.} \tag{*}%
\end{align}
(in particular, for $n=1$, we put $I[\varphi,\varphi^{+}]\searrow I[\alpha
_{1},\beta_{1}]\nearrow I[\psi,\psi^{+}]$).
\smallskip

In the sequel the following theorem from \cite[Lemma 22]{IdzSlo01} will be useful.

\begin{theorem}
\label{ddecomposition}Let $\mathbf{A}$ be an algebra with modular congruence
lattice. Then for all $\alpha,\beta\in\mathsf{Con}\left(  \mathbf{A}\right)  $ and
$\eta\in\mathsf{Cm}\left(  \mathbf{A}\right)  $ with $\alpha\wedge\beta
\leq\eta$ there are $\mu_{1},\mu_{2}\in\eta/\!\sim\cup\left\{  \mathbf{1}%
_{\mathbf{A}}\right\}  $ such that $\alpha\leq\mu_{1},\beta\leq\mu_{2}$ and
$\mu_{1}\wedge\mu_{2}\leq\eta$, where $\mathbf{1}_{\mathbf{A}}$ denotes the
largest element in $\mathsf{Con}\left(  \mathbf{A}\right)  $.
\end{theorem}

Applying this theorem, we see that the relation PIP can be used to
characterize distributivity property of the congruence lattice of $\mathbf{A}$.

\begin{proposition}
\label{distributive}The following conditions are equivalent for an algebra
$\mathbf{A}$ with modular congruence lattice:

\begin{enumerate}
\item[(i)] $\mathsf{Con}\left(  \mathbf{A}\right)  $ is distributive;

\item[(ii)] $\varphi/\!\sim\,=\left\{  \varphi\right\}  $ for every
$\varphi\in\mathsf{Cm}\left(  \mathbf{A}\right)  $.
\end{enumerate}
\end{proposition}

\begin{proof}
$\left(  i\right)  \Rightarrow\left(  ii\right)  $ Let $\psi\in\varphi
/\!\!\sim$. Observe that for $\nu,\mu\in\mathsf{Cm}\left(  \mathbf{A}\right)
$, $\alpha,\beta\in\mathsf{Con}\left(  \mathbf{A}\right)  $ from $I[\nu
,\nu^{+}]\searrow I[\alpha,\beta]\nearrow I[\mu,\mu^{+}]$ it follows that
$\nu=\mu$. Indeed, we have $\beta\wedge\mu=\alpha\leq\nu$, and, by (i),
$\nu=(\beta\vee\nu)\wedge(\mu\vee\nu)=\nu^{+}\wedge(\mu\vee\nu)$. Hence
$\mu\vee\nu=\nu$ and, analogously, $\mu\vee\nu=\mu$, as desired. Taking this
observation into account, we get from (*) the equalities $\varphi=\mu
_{1}=\ldots=\mu_{n-1}=\psi$.

$(ii)\Rightarrow(i)$ Let $\alpha,\beta,\gamma\in\mathsf{Con}\left(
\mathbf{A}\right)  $. To show that $\alpha\wedge\left(  \beta\vee
\gamma\right)  =\left(  \alpha\wedge\beta\right)  \vee\left(  \alpha
\wedge\gamma\right)  $ it suffices to prove that $\left(  \alpha\wedge
\beta\right)  \vee\left(  \alpha\wedge\gamma\right)  \leq\mu$ implies
$\alpha\wedge\left(  \beta\vee\gamma\right)  \leq\mu$ for every $\mu
\in\mathsf{Cm}\left(  \mathbf{A}\right)  $. If $\mu\in\mathsf{Cm}\left(
\mathbf{A}\right)  $ and $\left(  \alpha\wedge\beta\right)  \vee\left(
\alpha\wedge\gamma\right)  \leq\mu$, then $\alpha\wedge\beta,\alpha
\wedge\gamma\leq\mu$. From $\mu/\!\sim\,=\left\{  \mu\right\}  $ and from
Theorem \ref{ddecomposition} we deduce that either $\alpha\leq\mu$ or
$\beta\leq\mu$ and either $\alpha\leq\mu$ or $\gamma\leq\mu$. Hence either
$\alpha\leq\mu$ or $\beta\vee\gamma\leq\mu$, and so $\alpha\wedge\left(
\beta\vee\gamma\right)  \leq\mu$.
\end{proof}

From now on in this section we strengthen our assumption on $\mathbf{A}$ by
requiring that $\mathbf{A}$ belongs to a congruence modular variety. In
particular, this means that we can apply the commutator theory of Freese and
McKenzie \cite{FreMcK87} to study the properties of $\mathbf{A}$. The
definition of the relation PIP simplifies considerably in the case of algebras
satisfying the condition (SC1) that was isolated and studied by Idziak and
S\l omczy\'{n}ska in \cite{IdzSlo01}.

\begin{definition}
An algebra $\mathbf{A}$ from a congruence modular variety fulfills
the \emph{condition $\mathrm{(SC1)}$} if and only if for every $\varphi\in
\mathsf{Cm}\left(  \mathbf{A}\right)  $ the centralizer $(\varphi:\varphi
^{+})\leq\varphi^{+}$.
\end{definition}

Note that this condition implies the condition (C1) from \cite{FreMcK87}
described by a commutator identity: $[\alpha,\beta]=([\alpha,\alpha
]\wedge\beta)\vee([\beta,\beta]\wedge\alpha)$ for every $\alpha,\beta
\in\mathsf{Con}\left(  \mathbf{A}\right)  $, see \cite[Theorem 15]{IdzSlo01}.
We will see further that this class covers many interesting cases, in
particular Fregean algebras that play an important role in the algebraization
of intuitionistic logic and its fragments \cite{IdzSloWro09}. From the
definition of \textrm{(SC1)} we get the following lemma that will be
frequently used throughout the paper, see also \cite[Lemma 21]{IdzSlo01}.

\begin{lemma}
\label{centralmon}Let $\mathbf{A}$ be an algebra from a congruence modular
variety and $\varphi\in\mathsf{Cm}\left(  \mathbf{A}\right)  $. Then

\begin{enumerate}
\item if $\varphi^{+}$ is not Abelian over $\varphi$, then $\left(
\varphi:\varphi^{+}\right)  =\varphi$, and $\varphi/\!\!\sim\,=\left\{
\varphi\right\}  $.

\item if $\mathbf{A}$ fulfills \textrm{(SC1)} and $\varphi^{+}$ is Abelian
over $\varphi$, then

\begin{enumerate}
\item $\left(  \varphi:\varphi^{+}\right)  =\varphi^{+}$, and

\item $\psi^{+}$ is Abelian over $\psi$ and $\psi^{+}=\varphi^{+}$ for every
$\psi\in\varphi\!/\!\!\sim$.
\end{enumerate}
\end{enumerate}
\end{lemma}

\begin{proof}
We have $\varphi\leq\left(  \varphi:\varphi^{+}\right)  $. Note that
$\varphi^{+}$ is Abelian over $\varphi$ if and only if $\left(  \varphi
:\varphi^{+}\right)  \geq\varphi^{+}$. Since projectivity preserves
centralizers and abelianity \cite{FreMcK87}, for $\psi\in\mathsf{Cm}\left(
\mathbf{A}\right)  $, $\psi\sim\varphi$, we deduce that $(\psi:\psi
^{+})=(\varphi:\varphi^{+})$ and $\psi^{+}$ is Abelian over $\psi$ if and only
if $\varphi^{+}$ is Abelian over $\varphi$. This implies (1). If $\mathbf{A}$
fulfills \textrm{(SC1)} and $\varphi^{+}$ is Abelian over $\varphi$, we get
$\psi^{+}=(\psi:\psi^{+})=(\varphi:\varphi^{+})=\varphi^{+}$, and (2) follows.
\end{proof}

Assume that an algebra $\mathbf{A}$ from a congruence modular variety
satisfies \textrm{(SC1)}.

\begin{definition}
For $U\in\mathsf{Cm}\left(  \mathbf{A}\right)  /\!\!\sim$ we put
$0_{U}:=\bigwedge U$. For $S\subset U$, set $\overline{S}:=S\cup\left\{
\eta^{+}\right\}  $, where $\eta\in U$. (It follows from Lemma
\ref{centralmon}.2 that the choice of $\eta$ is irrelevant.)
\end{definition}

In this situation, we can easily characterized these congruences $\alpha$ for
which $\left[  \alpha,\alpha\right]  \leq0_{U}$.

\begin{proposition}
\label{abeovezer}Suppose that an algebra $\mathbf{A}$ from a congruence
modular variety satisfies \textrm{(SC1)}. Let $U=\eta/\!\!\sim\;\in
\mathsf{Cm}\left(  \mathbf{A}\right)  /\!\!\sim$. For $\alpha\in
\mathsf{Con}\left(  \mathbf{A}\right)  $ we have:

1. if $\left[  \alpha,\alpha\right]  \leq0_{U}$, then $\alpha\leq\eta^{+}$;

2. if $\eta^{+}$ is Abelian over $\eta$, then $\left[  \alpha,\alpha\right]
\leq0_{U}$ iff $\alpha\leq\eta^{+}$.
\end{proposition}

\begin{proof}
(1) Let $\left[  \alpha,\alpha\right]  \leq0_{U}$. Applying the definition of
\textrm{(SC1)}, we get $\left(  \eta:\eta^{+}\right)  \leq\eta^{+}$. Suppose
that $\alpha\nleqslant\eta^{+}$. Hence $\left[  \eta^{+},\alpha\vee
\eta\right]  \nleqslant\eta$. Then $\eta^{+}=\eta\vee\left[  \eta^{+}%
,\alpha\vee\eta\right]  \leq\eta\vee\left[  \alpha\vee\eta,\alpha\vee
\eta\right]  =\eta\vee\left[  \alpha,\alpha\right]  =\eta$, a contradiction.

(2) Now, let $\eta^{+}$ is Abelian over $\eta$ and $\alpha\leq\eta^{+}$. For
$\psi\in\eta/\!\!\sim$, it follows from Lemma \ref{centralmon}.2, that
$\psi^{+}=\eta^{+}$ is Abelian over $\psi$. In consequence, $\left[  \eta
^{+},\eta^{+}\right]  \leq\psi$. Thus $\left[  \alpha,\alpha\right]
\leq\left[  \eta^{+},\eta^{+}\right]  \leq0_{U}$.
\end{proof}

Observe that from Lemma \ref{centralmon}.2 it follows immediately that two
different congruences $\varphi,\psi\in\mathsf{Cm}\left(  \mathbf{A}\right)  $
such that $\varphi\sim\psi$ fulfills $\varphi^{+}=\psi^{+}$ and so are
incomparable. This observation allows us to obtain the projectivity of
completely meet irreducible congruences in (*) in short chains.

\begin{lemma}
\label{triplearrow}Suppose that an algebra $\mathbf{A}$ from a congruence
modular variety satisfies \textrm{(SC1)}. Then for all $\varphi,\psi
\in\mathsf{Cm}(\mathbf{A})$ such that $\varphi\sim\psi$ there exist
$\alpha,\beta\in\mathsf{Con}(\mathbf{A})$ with $I[\varphi,\varphi^{+}]\searrow
I[\alpha,\beta]\nearrow I[\psi,\psi^{+}]$.
\end{lemma}

\begin{proof}
Let $\varphi,\psi\in\mathsf{Cm}(\mathbf{A})$, $\varphi\sim\psi$. We assume
that $\varphi\neq\psi$, since otherwise the statement is obvious. Then, there
exist $\alpha_{1},\ldots,\alpha_{n},\beta_{1},\ldots,\beta_{n}\in
\mathsf{Con}(\mathbf{A})$ and $\mu_{1},\ldots,\mu_{n-1}\in\mathsf{Cm}%
(\mathbf{A})$ fulfilling (*). To reduce this chain to a short one we may
induct on $n$, the number of up-arrows. Actually, it suffices to consider the
case $n=2$. Therefore suppose that
\[
I[\varphi,\varphi^{+}]\searrow I[\alpha_{1},\beta_{1}]\nearrow I[\mu,\mu
^{+}]\searrow I[\alpha_{2},\beta_{2}]\nearrow I[\psi,\psi^{+}]
\]
holds for some $\alpha_{1},\beta_{1},\alpha_{2},\beta_{2}\in\mathsf{Con}%
(\mathbf{A})$ and $\mu\in\mathsf{Cm}(\mathbf{A})$.

In consequence, it follows from Lemma~\ref{centralmon}, that (1) $\varphi
^{+}=\mu^{+}=\psi^{+}$. Then we may assume that (2) $\beta_{2}\leq\varphi
\vee\alpha_{2}$ and (3) $\beta_{1}\leq\psi\vee\alpha_{1}$, since otherwise we
would get either $I[\varphi,\varphi^{+}]\searrow I[\alpha_{2},\beta_{2}]$ or
$I[\alpha_{1},\beta_{1}]\nearrow I[\psi,\psi^{+}]$, respectively. Now, we will
show that the congruences%
\[
\alpha:=(\varphi\wedge\alpha_{2})\vee(\psi\wedge\alpha_{1})
\]
and%
\[
\beta:=\mu\wedge((\varphi\wedge\beta_{2})\vee(\psi\wedge\beta_{1}))
\]
witness our Lemma. Put $\gamma:=(\varphi\wedge\beta_{2})\vee(\psi\wedge
\beta_{1})$. Using modularity $(m)$ we get
\begin{gather*}
\varphi\vee\beta=\varphi\vee(\varphi\wedge\beta_{2})\vee(\mu\wedge
\gamma)\underset{(m)}{=}\varphi\vee\lbrack((\varphi\wedge\beta_{2})\vee
\mu)\wedge\gamma]\\
=\varphi\vee\lbrack((\varphi\wedge\beta_{2})\vee\alpha_{2}\vee\mu)\wedge
\gamma]\underset{(m)}{=}\varphi\vee\lbrack(((\varphi\vee\alpha_{2})\wedge
\beta_{2})\vee\mu)\wedge\gamma]\\
\underset{(2)}{=}\varphi\vee\lbrack(\beta_{2}\vee\mu)\wedge\gamma]=\varphi
\vee\lbrack\mu^{+}\wedge\gamma]\underset{(1)}{=}\varphi\vee\lbrack\varphi
^{+}\wedge\gamma]\\
\underset{(1)}{=}\varphi\vee\gamma=\varphi\vee(\psi\wedge\beta_{1}%
)=\varphi\vee\alpha_{1}\vee(\psi\wedge\beta_{1})\\
\underset{(m)}{=}\varphi\vee((\alpha_{1}\vee\psi)\wedge\beta_{1}%
)\underset{(3)}{=}\varphi\vee\beta_{1}=\varphi^{+}\text{,}%
\end{gather*}
and in the very same way $\psi\vee\beta=\psi^{+}$. Moreover
\begin{align*}
\varphi\wedge\beta &  =\mu\wedge\varphi\wedge((\varphi\wedge\beta_{2}%
)\vee(\psi\wedge\beta_{1}))\\
&  \underset{(m)}{=}\mu\wedge((\varphi\wedge\beta_{2})\vee(\varphi\wedge
\psi\wedge\beta_{1}))\\
&  =\mu\wedge((\varphi\wedge\beta_{2})\vee(\psi\wedge\alpha_{1}))\\
&  \underset{(m)}{=}(\mu\wedge\varphi\wedge\beta_{2})\vee(\psi\wedge\alpha
_{1})\\
&  =(\varphi\wedge\alpha_{2})\vee(\psi\wedge\alpha_{1})=\alpha,
\end{align*}
and analogously $\psi\wedge\beta=\alpha$, which completes the proof.
\end{proof}

The main focus of this work is to understand the behaviour of the relation PIP
for an algebra $\mathbf{A}$ from a congruence modular variety such that
$\mathrm{H}(\mathbf{A})$ is Fregean, see Sec. \ref{Fregean}. In particular, we
show that in this case each equivalence class $\eta/\!\!\sim$ in
$\mathsf{Cm}\left(  \mathbf{A}\right)  $, supplemented with a unit element
$\eta^{+}$, forms a Boolean group.

\section{PIP in Fregean algebras\label{Fregean}}

In the sequel we assume that the language of an algebra (or class of algebras
of the same type) we consider contains a distinguished constant term $1$.

\begin{definition}
We call an algebra $\mathbf{A}$ \emph{congruence orderable} (with respect to
$1$) if $\Theta_{\mathbf{A}}\left(  1,a\right)  =\Theta_{\mathbf{A}}\left(
1,b\right)  $ implies $a=b$ for $a,b\in A$, where $\Theta_{\mathbf{A}}\left(
c,d\right)  $ denotes the smallest congruence containing $\left(  c,d\right)
$. We call it \emph{Fregean} if, additionally, $1/\alpha=1/\beta$ implies
$\alpha=\beta$ for $\alpha,\beta\in\mathsf{Con}\left(  \mathbf{A}\right)  $
(i.e., $\mathbf{A}$ is $1$\emph{-regular}).

Thus, if an algebra is congruence orderable, then the formula $a\leq b$ iff
$\Theta_{\mathbf{A}}\left(  1,a\right)  \subset\Theta_{\mathbf{A}}\left(
1,b\right)  $ for $a,b\in A$ introduces a partial order in $A$, and, if it is
Fregean, then the congruences are uniquely determined by their $1$%
--equivalence classes. A class of algebras of the same type is called
\emph{Fregean} (\emph{congruence orderable}) if all algebras in this class are
Fregean (congruence orderable) with respect to the common constant term $1$.
\end{definition}

The following property of subdirectly irreducible congruence orderable
algebras is crucial for understanding the structure of Fregean algebras.

\begin{proposition}
\label{natord-mu}\cite[Lemma 2.1]{IdzSloWro09} If $\mu$ is the monolith of a
subdirectly irreducible congruence orderable (with respect to $1$) algebra
$\mathbf{A}$, then $|1/\mu|=2$ and all other $\mu$--cosets are one element.
\end{proposition}

Hence, if $\mathbf{A}$ is a subdirectly irreducible congruence orderable
algebra, then there is the unique non-unit element in the unique non trivial
coset of the monolith of $\mathbf{A}$. Clearly, this element is the largest
non-unit element in $A$. If $\mathbf{A}$ is Fregean, then the reverse is true:
if there is the largest non-unit element in $A$, then $\mathbf{A}$ is
subdirectly irreducible.

Observe that for an algebra $\mathbf{A}$, it follows directly from the
definition of orderability that $\mathrm{H}(\mathbf{A})$, the class of all
homomorphic images (quotient algebras) of $\mathbf{A}$, is congruence
orderable if and only if for every $a,b\in A$ and $\alpha\in\mathsf{Con}%
\left(  \mathbf{A}\right)  $ we have $\alpha\vee\Theta_{\mathbf{A}}\left(
1,a\right)  =\alpha\vee\Theta_{\mathbf{A}}\left(  1,b\right)  $ iff
$(a,b)\in\alpha$ \cite[Lemma 3.3]{IdzSlo01}. What is more, we can characterize
this property in a simple way.

\begin{proposition}
\label{altern}The following conditions are equivalent:

\begin{enumerate}
\item $\mathrm{H}(\mathbf{A})$ is congruence orderable with respect to $1$;

\item for all $\varphi\in\mathsf{Cm}\left(  \mathbf{A}\right)  $ and $a,b\in
A$
\[
\left(  a,b\right)  \in\varphi^{+}\text{ implies }\left\{  \left(  a,b\right)
,\left(  1,a\right)  ,\left(  1,b\right)  \right\}  \cap\varphi\neq
\emptyset\text{.}%
\]

\end{enumerate}
\end{proposition}

\begin{proof}
$(1)\Rightarrow(2)$. Let $\varphi\in\mathsf{Cm}\left(  \mathbf{A}\right)  $,
$a,b\in A$ and $\left(  a,b\right)  \in\varphi^{+}$. Then $\mathbf{A}/\varphi$
is subdirectly irreducible with the monolith $\varphi^{+}/\varphi$. Assuming
that $\left(  a,b\right)  \notin\varphi$, we get $\{a/\varphi,b/\varphi\}$ is
a two-element and $(a/\varphi,b/\varphi)\in\varphi^{+}/\varphi$. By
Proposition \ref{natord-mu}, $1/\varphi\in\{a/\varphi,b/\varphi\}$, and so
$\left(  1,a\right)  \in\varphi$ or $\left(  1,b\right)  \in\varphi$.

$(2)\Rightarrow(1)$. Choose $\alpha\in\mathsf{Con}\left(  \mathbf{A}\right)  $
and $a,b\in A$ such that $\alpha\vee\Theta_{\mathbf{A}}\left(  1,a\right)
=\alpha\vee\Theta_{\mathbf{A}}\left(  1,b\right)  $. Suppose that $\left(
a,b\right)  \notin\alpha$. Then we can find $\varphi\in\mathsf{Cm}\left(
\mathbf{A}\right)  $ such that $\alpha\leq\varphi$ and $\left(  a,b\right)
\in\varphi^{+}\backslash\varphi$. Thus $\varphi\vee\Theta_{\mathbf{A}}\left(
1,a\right)  =\varphi\vee\Theta_{\mathbf{A}}\left(  1,b\right)  $, and either
$\left(  1,a\right)  \in\varphi$ or $\left(  1,b\right)  \in\varphi$. Hence
$\left(  a,b\right)  \in\varphi$, a contradiction.
\end{proof}

Moreover, observe that if $\mathbf{A}$ is $1$-regular, then it is easy to
check that $\mathrm{H}(\mathbf{A})$ is $1-$regular as well. In consequence,
$\mathrm{H}(\mathbf{A})$ is Fregean iff $\mathrm{H}(\mathbf{A})$ is congruence
orderable and $\mathbf{A}$ is $1$-regular. It follows from \cite{Hag73} that
$1$-regular varieties (and so Fregean varieties) are congruence modular. While
we need congruence modularity in our applications, the assumption that the
whole variety generated by $\mathbf{A}$ is $1$-regular seems to be too strong.
Thus, in the sequel, we shall assume that $\mathrm{H}(\mathbf{A})$ is Fregean,
but the variety generated by $\mathbf{A}$ is just congruence modular. From
\cite[Theorem 2.3]{IdzSloWro09} it follows that all such algebras satisfy the
condition (SC1):

\begin{theorem}
\label{FreSC1}Let $\mathbf{A}$ be an algebra from a congruence modular variety
and let $\mathrm{H}(\mathbf{A})$ is Fregean. Then $\mathbf{A}$ satisfies
$(SC1)$.
\end{theorem}

Thus, it seems natural to distinguish this class of algebras.

\begin{definition}
We call an algebra $\mathbf{A}$ \emph{strongly Fregean} if:

\begin{itemize}
\item $\mathrm{H}(\mathbf{A})$ is Fregean;

\item $\mathcal{V}(\mathbf{A})$, the variety generated by $\mathbf{A}$, is
congruence modular.
\end{itemize}
\end{definition}

Note that if $\mathbf{A}$ is from Fregean variety, then $\mathbf{A}$ is
strongly Fregean. However, the reverse implication is not true. To show this,
it is enough to consider two-element lattice with the greatest element $1$.
From now on, unless otherwise stated, we shall assume that $\mathbf{A}$ is
strongly Fregean.

Now let us go back to the study of the relation PIP. We will show that each
equivalence class $\eta/\!\!\sim$ in $\mathsf{Cm}\left(  \mathbf{A}\right)  $,
supplemented with a unit element $\eta^{+}$, forms a Boolean group with the
complement of symmetric difference restricted to $\eta^{+}$.

Let $\varphi,\psi\in\mathsf{Cm}\left(  \mathbf{A}\right)  $ and $\varphi
^{+}=\psi^{+}$. Define
\[
\varphi\bullet\psi:=\left(  \varphi\div\psi\right)  ^{\prime}\cap\varphi
^{+}\text{.}%
\]

\begin{theorem}
\label{Boolean}Let $\mathbf{A}$ be strongly Fregean and let $U\in
\mathsf{Cm}\left(  \mathbf{A}\right)  /\!\!\sim$. Then $\left(  \overline
{U},\bullet\right)  $ is a Boolean group.
\end{theorem}

\begin{proof}
Let $\eta\in\mathsf{Cm}\left(  \mathbf{A}\right)  $ such that $U=\eta
/\!\!\sim$, and $\varphi,\psi\in U$. It follows from Lemma \ref{centralmon}.2
that $\varphi^{+}=\psi^{+}=\eta^{+}$. It suffices to show that $\varphi
\bullet\psi\in U$ for $\varphi\neq\psi$, as $\varphi\bullet\varphi=\eta^{+}$.
By Theorem \ref{FreSC1} and Lemma \ref{triplearrow}, there exist $\alpha
,\beta\in\mathsf{Con}\left(  \mathbf{A}\right)  $ such that $I\left[
\varphi,\eta^{+}\right]  \searrow I\left[  \alpha,\beta\right]  \nearrow
I\left[  \psi,\eta^{+}\right]  $. Put $\gamma:=\left(  \varphi\wedge
\psi\right)  \vee\beta$. It is enough to show that $\varphi\bullet\psi=\gamma$
and $\gamma\in U$.

Step 1. Clearly $\alpha\leq\varphi\wedge\psi$, and so, by modularity,
$\varphi\wedge\gamma=\left(  \varphi\wedge\psi\right)  \vee(\varphi\wedge
\beta)=\left(  \varphi\wedge\psi\right)  \vee\alpha=\varphi\wedge\psi$.
Analogously, $\psi\wedge\gamma=\varphi\wedge\psi$. Hence $\gamma\subset
\varphi\bullet\psi$, since $\varphi\wedge\psi=\varphi\wedge\gamma=\psi
\wedge\gamma$. Moreover $\varphi\vee\gamma=\eta^{+}=\psi\vee\gamma$.
Consequently, $I\left[  \varphi,\eta^{+}\right]  \searrow I\left[
\varphi\wedge\psi,\gamma\right]  \nearrow I\left[  \psi,\eta^{+}\right]  $ and
$I\left[  \varphi,\eta^{+}\right]  \searrow I\left[  \varphi\wedge\psi
,\psi\right]  \nearrow I\left[  \gamma,\eta^{+}\right]  $. Now, we show that
$\gamma\supset\varphi\bullet\psi$. On the contrary, suppose that there exists
$\left(  x,y\right)  \in\left(  \varphi\bullet\psi\right)  \backslash\gamma$.
Then $\left(  x,y\right)  \in\eta^{+}\backslash(\varphi\cup\psi)$. Applying
Proposition \ref{altern} we deduce that $\operatorname*{card}\left(  \left\{
\left(  1,x\right)  ,\left(  1,y\right)  \right\}  \cap\varphi\right)
=\operatorname*{card}\left(  \left\{  \left(  1,x\right)  ,\left(  1,y\right)
\right\}  \cap\psi\right)  =1$, and so $\left(  1,x\right)  ,\left(
1,y\right)  \in\eta^{+}$. In consequence, $\gamma\vee\Theta_{\mathbf{A}%
}\left(  1,x\right)  ,\gamma\vee\Theta_{\mathbf{A}}\left(  1,y\right)  \in
I\left[  \gamma,\eta^{+}\right]  $. From projectivity, we have
$\operatorname*{card}I\left[  \gamma,\eta^{+}\right]  =\operatorname*{card}%
I\left[  \varphi,\eta^{+}\right]  =2$. From orderability of $\mathbf{A/}%
\gamma\in\mathrm{H}(\mathbf{A})$ we get $\gamma\vee\Theta_{\mathbf{A}}\left(
1,x\right)  \neq\gamma\vee\Theta_{\mathbf{A}}(1,y)$, since otherwise $\left(
x,y\right)  \in\gamma$, a contradiction. Thus either $\gamma\vee
\Theta_{\mathbf{A}}\left(  1,x\right)  =\gamma$ or $\gamma\vee\Theta
_{\mathbf{A}}\left(  1,y\right)  =\gamma$. Without loss of generality we can
assume that $\gamma\vee\Theta_{\mathbf{A}}\left(  1,x\right)  =\gamma$, and so
$\left(  1,x\right)  \in\gamma$. Then $\left(  1,y\right)  \notin\gamma$.
Hence $\left(  1,x\right)  \in\gamma\wedge\varphi\wedge\psi$, and so $\left(
1,y\right)  \notin\gamma\cup\varphi\cup\psi$. From $1$-regularity there exists
$s\in A$ such that $\left(  1,s\right)  \in\gamma\backslash\left(
\varphi\wedge\psi\right)  $. Then $\left(  1,s\right)  \notin\varphi\cup\psi$,
and we can apply again Proposition \ref{altern} obtaining $\left(  s,y\right)
\in\varphi\wedge\psi\leq\gamma$. Hence $\left(  1,y\right)  \in\gamma$, a contradiction.

Step 2. We show that $\gamma\in U$. Let $\gamma=\bigwedge\Gamma$, where
$\Gamma\subset\mathsf{Cm}\left(  \mathbf{A}\right)  $. Since
$\operatorname*{card}I\left[  \gamma,\eta^{+}\right]  =2$, there exists
$\mu\in\Gamma$ such that $\eta^{+}\nleq\mu$, and so $\gamma=\eta^{+}\wedge\mu
$. Hence $I\left[  \gamma,\eta^{+}\right]  \nearrow I\left[  \mu,\eta^{+}%
\vee\mu\right]  $. From $\operatorname*{card}I\left[  \mu,\eta^{+}\vee
\mu\right]  =2$ we get $\mu^{+}=\eta^{+}\vee\mu$. As the intervals $I\left[
\varphi,\eta^{+}\right]  $, $I\left[  \gamma,\eta^{+}\right]  $ are
projective, we deduce that $\mu\sim\varphi$. Lemma \ref{centralmon}.1 gives us
$\eta^{+}=\mu^{+}$. Thus $\gamma=\mu^{+}\wedge\mu=\mu\in U$, as desired.
\end{proof}

\begin{corollary}
\label{iff}Let $\mathbf{A}$ be strongly Fregean, $\varphi,\psi\in
\mathsf{Cm}\left(  \mathbf{A}\right)  $, $\varphi\neq\psi$ and $\varphi
^{+}=\psi^{+}$. Then the following conditions are equivalent:

\begin{enumerate}
\item $\varphi\sim\psi$;

\item $\varphi\bullet\psi\in\mathsf{Cm}\left(  \mathbf{A}\right)  $;

\item $\varphi\bullet\psi\in\mathsf{Con}\left(  \mathbf{A}\right)  $.
\end{enumerate}
\end{corollary}

\begin{proof}
$(1)\Rightarrow(2)$ follows form the theorem above, and $(2)\Rightarrow(3)$ is
obvious. To prove $(3)\Rightarrow(1)$, observe that there exist $a,b\in A$
such that $(1,a)\in\varphi\backslash\psi$ and $(1,b)\in\psi\backslash\varphi$.
Hence $(a,b)\notin\varphi\cup\psi$ and $(a,b)\in\varphi^{+}$. Consequently,
$(a,b)\in\varphi\bullet\psi$, and so $\left(  \varphi\bullet\psi\right)
\vee\varphi=\varphi^{+}=\psi^{+}=\left(  \varphi\bullet\psi\right)  \vee\psi$.
Thus $I\left[  \varphi,\varphi^{+}\right]  \searrow I\left[  \varphi\wedge
\psi,\varphi\bullet\psi\right]  \nearrow I\left[  \psi,\psi^{+}\right]  $, as desired.
\end{proof}

The following lemma and theorem describe the location of the equivalence
classes of the relation PIP in the congruence lattice.

\begin{lemma}
\label{atom}Let $\mathbf{A}$ be strongly Fregean, $\mu\in\mathsf{Cm}\left(
\mathbf{A}\right)  $, $U=\mu/\!\!\sim$, and $a\in A$. If $\left(  1,a\right)
\in\mu^{+}\backslash0_{U}$, then $\Theta_{\mathbf{A}}\left(  1,a\right)
\vee0_{U}$ is an atom in $\left[  0_{U},\mu^{+}\right]  $.
\end{lemma}

\begin{proof}
Since $\left(  1,a\right)  \notin0_{U}$, it follows that there exists
$\varphi\in U$ such that $\left(  1,a\right)  \notin\varphi$. Clearly, from
Lemma \ref{centralmon}.2 we have $\varphi^{+}=\mu^{+}$. We show that
$\varphi\wedge\Theta_{\mathbf{A}}\left(  1,a\right)  \leq0_{U}$. On the
contrary, suppose that there exists $(x,y)\in\varphi\wedge\Theta_{\mathbf{A}%
}\left(  1,a\right)  $ and $(x,y)\notin0_{U}$. Hence there exists $\psi\in U$
such that $(x,y)\notin\psi$. In consequence, $\psi\neq\varphi$ and $\left(
1,a\right)  \notin\psi$. From Theorem \ref{Boolean} we get $\left(
1,a\right)  \in\varphi\bullet\psi$, and so $(x,y)\in\varphi\wedge
(\varphi\bullet\psi)=\varphi\wedge\psi\leq\psi$, a contradiction. Thus, from
modularity, $\varphi\wedge(\Theta_{\mathbf{A}}\left(  1,a\right)  \vee
0_{U})=0_{U}$. In consequence, $I\left[  0_{U},\Theta_{\mathbf{A}}\left(
1,a\right)  \vee0_{U}\right]  \nearrow I\left[  \varphi,\varphi^{+}\right]  $.
Hence $\Theta_{\mathbf{A}}\left(  1,a\right)  \vee0_{U}$ covers $0_{U}$, as desired.
\end{proof}

\begin{theorem}
\label{location}Let $\mathbf{A}$ be strongly Fregean, $\mu\in\mathsf{Cm}%
\left(  \mathbf{A}\right)  $ and $U=\mu/\!\!\sim$. Then

\begin{enumerate}
\item $U=\left\{  \beta\in\mathsf{Con}\left(  \mathbf{A}\right)  :0_{U}%
\leq\beta\prec\mu^{+}\right\}  $ ;

\item $U=\mathsf{Cm}\left(  \mathbf{A}\right)  \cap I\left[  0_{U},\mu
^{+}\right)  $ ;

\item If $0_{U}\leq\alpha\in\mathsf{Con}\left(  \mathbf{A}\right)  $, then
either $\alpha\leq\mu^{+}$ or $\mu^{+}\leq\alpha$.
\end{enumerate}
\end{theorem}

\begin{proof}
1. Clearly, if $\beta\in U$, then $0_{U}\leq\beta\prec\mu^{+}$. Let $\beta
\in\mathsf{Con}\left(  \mathbf{A}\right)  $ be such that $0_{U}\leq\beta
\prec\mu^{+}$. If we pick $\varphi\in\mathsf{Cm}\left(  \mathbf{A}\right)  $
over $\beta$ but not over $\mu^{+}$, then $\mu^{+}\wedge\varphi=\beta$, and so
$I\left[  \beta,\mu^{+}\right]  \nearrow I\left[  \varphi,\varphi^{+}\right]
$. From $1$-regularity we deduce that there is $x\in A$ such that $\left(
1,x\right)  \in\mu^{+}\backslash\beta$. Clearly, $\left(  1,x\right)
\notin0_{U}$, and so there exists $\nu\in U$ such that $\left(  1,x\right)
\notin\nu$. Put $\gamma=\Theta_{\mathbf{A}}\left(  1,x\right)  \vee0_{U}$. It
follows from Lemma \ref{atom} that $0_{U}\prec\gamma$. Then $I\left[  \nu
,\nu^{+}\right]  \searrow I\left[  0_{U},\gamma\right]  \nearrow I\left[
\beta,\mu^{+}\right]  \nearrow I\left[  \varphi,\varphi^{+}\right]  $. Hence
$\varphi\sim\nu$ and $\varphi^{+}=\nu^{+}=\mu^{+}$. Thus $\beta=\mu^{+}%
\wedge\varphi=\varphi\in U$.

2. If $\operatorname*{card}U=1$, then the equality is obvious. Let us assume
that $\operatorname*{card}U\geq2$. Take $\gamma\in\mathsf{Cm}\left(
\mathbf{A}\right)  \cap I\left[  0_{U},\mu^{+}\right)  $. Clearly, $\gamma
^{+}\leq\mu^{+}$. From Lemma \ref{centralmon} we deduce that $\mu^{+}$ is
Abelian over $\mu$. Consequently, by Proposition \ref{abeovezer}, $\mu^{+}$ is
Abelian over $0_{U}$. We get $\left[  \gamma^{+},\mu^{+}\right]  \leq\left[
\mu^{+},\mu^{+}\right]  \leq0_{U}\leq\gamma$, and so $\mu^{+}\leq\left(
\gamma:\gamma^{+}\right)  $. On the other hand, it follows from Theorem
\ref{FreSC1} that $\left(  \gamma:\gamma^{+}\right)  \leq\gamma^{+}$. Thus
$\mu^{+}\leq\gamma^{+}$, and so $\mu^{+}=\gamma^{+}$. In consequence,
$0_{U}\leq\gamma\prec\mu^{+}$ and from (1) we get $\gamma\in U$.

3. Let $\alpha\in\mathsf{Con}\left(  \mathbf{A}\right)  $, $0_{U}\leq\alpha$.
As in the proof of (2) we consider two cases: $\operatorname*{card}U=1$ and
$\operatorname*{card}U\geq2$. In the former case, we have $U=\left\{
\mu\right\}  $, and so $\mu=0_{U}\leq\alpha$. Hence, either $\alpha=\mu
<\mu^{+}$, or $\mu<\alpha$, and so $\mu^{+}\leq\alpha$. In the latter case, we
see from Lemma \ref{centralmon} that $\mu^{+}$ is Abelian over $\mu$. Assume
that $\mu^{+}\nleq\alpha$. Then $0_{U}\leq\alpha\wedge\mu^{+}<\mu^{+}$. We
choose $\eta\in\mathsf{Cm}\left(  \mathbf{A}\right)  $ such that $\alpha
\wedge\mu^{+}\leq\eta$ and $\mu^{+}\nleq\eta$. Since $\left[  \mu^{+},\mu
^{+}\right]  \leq0_{U}\leq\eta$, we have $\left[  \mu^{+}\vee\eta,\eta
^{+}\right]  \leq\left[  \mu^{+}\vee\eta,\mu^{+}\vee\eta\right]  \leq\eta$.
Thus $\mu^{+}\vee\eta\leq(\eta:\eta^{+})$. Using (SC1) condition, we get
$(\eta:\eta^{+})\leq\eta^{+}$, and so $\mu^{+}\vee\eta\leq\eta^{+}$.
Consequently, $\mu^{+}\vee\eta=\eta^{+}$, and so $I[\mu^{+}\wedge\eta,\mu
^{+}]\nearrow I[\eta,\eta^{+}]$. Then $0_{U}\leq\mu^{+}\wedge\eta\prec\mu^{+}%
$, and from (1) we deduce that $\mu^{+}\wedge\eta\in U\subset\mathsf{Cm}%
\left(  \mathbf{A}\right)  $. Hence $\eta=\mu^{+}\wedge\eta\in U$ and
$\eta^{+}=\mu^{+}$. Then, from modularity, $\eta=\eta\vee\left(  \alpha
\wedge\mu^{+}\right)  =\eta\vee\left(  \alpha\wedge\eta^{+}\right)  =\left(
\eta\vee\alpha\right)  \wedge\eta^{+}$, and so $\alpha\leq\eta<\eta^{+}%
=\mu^{+}$.
\end{proof}

Applying Theorem \ref{FreSC1}, we can significantly strengthen Theorem
\ref{ddecomposition}. We start from the observation describing the relation
between the Boolean operation and meet operation in a given coset of the
relation PIP.

\begin{lemma}
\label{triple}Let $\mathbf{A}$ be strongly Fregean, $U\in\mathsf{Cm}\left(
\mathbf{A}\right)  /\!\!\sim$ and $\eta,\nu_{1},\nu_{2}\in U$, $\eta\neq
\nu_{1}\neq\nu_{2}\neq\eta$. Then $\nu_{1}\bullet\nu_{2}=\eta$ if and only if
$\nu_{1}\wedge\nu_{2}\leq\eta$.
\end{lemma}

\begin{proof}
Clearly $\nu_{1}\wedge\nu_{2}\leq\nu_{1}\bullet\nu_{2}$. Assume that $\nu
_{1}\wedge\nu_{2}\leq\eta$. By Theorem \ref{Boolean} we know that $\nu
_{1}\bullet\nu_{2}\in U$. It is enough to show that $\eta\leq\nu_{1}\bullet
\nu_{2}$, since all elements in $U$ are incomparable. On the contrary, suppose
that there exists $\left(  a,b\right)  \in\eta\backslash\left(  \nu_{1}%
\bullet\nu_{2}\right)  $. Then either $\left(  a,b\right)  \in\nu
_{1}\backslash\nu_{2}$ or $\left(  a,b\right)  \in\nu_{2}\backslash\nu_{1}$.
There is no loss of generality in assuming that $\left(  a,b\right)  \in
\nu_{1}\backslash\nu_{2}$. Using modularity of $\mathsf{Con}\left(
\mathbf{A}\right)  $ we get $\nu_{1}=\nu_{1}\wedge\eta^{+}=\nu_{1}%
\wedge\left(  \left(  \nu_{1}\bullet\nu_{2}\right)  \vee\Theta_{\mathbf{A}%
}\left(  a,b\right)  \right)  =\left(  \nu_{1}\wedge\left(  \nu_{1}\bullet
\nu_{2}\right)  \right)  \vee\Theta_{\mathbf{A}}\left(  a,b\right)  =\left(
\nu_{1}\wedge\nu_{2}\right)  \vee\Theta_{\mathbf{A}}\left(  a,b\right)
\leq\eta$. Hence $\nu_{1}=\eta$, a contradiction.
\end{proof}

\begin{proposition}
\label{baslem}Let $\mathbf{A}$ be strongly Fregean, $\eta\in\mathsf{Cm}\left(
\mathbf{A}\right)  $, and $\alpha_{1},\ldots,\alpha_{n}%
\in\mathsf{Con}\left(  \mathbf{A}\right)  $ be such that $\alpha_{1}\wedge
\cdots\wedge\alpha_{n}\leq\eta$. Put $U:=\eta/\!\!\sim$. Then there exist $\mu_{1},\ldots,\mu_{n}\in
U\cup\left\{  \mathbf{1}_{\mathbf{A}}\right\}  $ such that $\alpha_{i}\leq
\mu_{i}$ for $i=1,\ldots,n$, $\mu_{1}\wedge\cdots\wedge\mu_{n}\leq\eta$, and
$\eta$ belongs to a subuniverse of $\left(  \overline{U},\bullet\right)  $
generated by $\left\{  \mu_{i}:1,\ldots,n\right\}  \backslash\left\{
\mathbf{1}_{\mathbf{A}}\right\}  $, i.e., $\eta=\mu_{i_{1}}\bullet
\cdots\bullet\mu_{i_{k}}$ for some $1\leq i_{1}<\ldots<i_{k}\leq n$.
\end{proposition}

\begin{proof}
The proof is by induction on $n$. For $n=1$ the assertion is obvious. Suppose
that $n\geq2$ and the assertion is true for every $k<n$. We can assume that
$\alpha_{i}\nleq\eta$ $\left(  i=1,\ldots,n\right)  $, since otherwise we put
$\mu_{i}=\eta$ and $\mu_{j}=\mathbf{1}_{\mathbf{A}}$ for $j\neq i$. It follows
from Lemma \ref{ddecomposition} that there exist $\upsilon_{1},\upsilon_{2}\in
U\cup\left\{  \mathbf{1}_{\mathbf{A}}\right\}  $ such that $\alpha_{1}%
\wedge\cdots\wedge\alpha_{n-1}\leq\upsilon_{1}$ and $\alpha_{n}\leq
\upsilon_{2}$ and $\upsilon_{1}\wedge\upsilon_{2}\leq\eta$. If $\upsilon
_{1}=\mathbf{1}_{\mathbf{A}}$ or $\upsilon_{1}=\upsilon_{2}$, or $\upsilon
_{2}=\eta$, then $\alpha_{n}\leq\eta$, a contradiction. From $\upsilon_{1}%
\neq\mathbf{1}_{\mathbf{A}}$, by the induction hypothesis, it follows that
there exist $\mu_{1},\ldots,\mu_{n-1}\in U\cup\left\{  \mathbf{1}_{\mathbf{A}%
}\right\}  $ such that $\alpha_{i}\leq\mu_{i}$ for $i=1,\ldots,n-1$, $\mu
_{1}\wedge\cdots\wedge\mu_{n-1}\leq\upsilon_{1}$, and $\upsilon_{1}$ belongs
to a subuniverse of $\left(  \overline{U},\bullet\right)  $ generated by
$\left\{  \mu_{i}:1,\ldots,n-1\right\}  \backslash\left\{  \mathbf{1}%
_{\mathbf{A}}\right\}  $. If $\upsilon_{1}=\eta$, which covers also the case
$\upsilon_{2}=\mathbf{1}_{\mathbf{A}}$, we can put $\mu_{n}=\mathbf{1}%
_{\mathbf{A}}$, and the assertion holds. If $\eta\notin\left\{  \upsilon
_{1},\upsilon_{2}\right\}  $, we get from Lemma \ref{triple} that
$\eta=\upsilon_{1}\bullet\upsilon_{2}$. Putting $\mu_{n}=\upsilon_{2}%
\neq\mathbf{1}_{\mathbf{A}}$, we obtain the assertion.
\end{proof}

By the \textsl{length} $\delta(L)$ of a finite modular lattice $L$ we mean the
length of an arbitrary maximal chain in $L$ (where the length of a finite
chain with $n+1$ element is defined to be $n$). This quantity can be
characterized with the help of the PIP relation.

\begin{theorem}
\label{dimension}Let $\mathbf{A}$ be finite and strongly Fregean. Then
\[
\delta\left(  \mathsf{Con}\left(  \mathbf{A}\right)  \right)  =\sum
\{\dim\overline{U}:U\in\mathsf{Cm}\left(  \mathbf{A}\right)  /\!\!\sim
\}\text{,}%
\]
where $\dim\overline{U}$ denotes the dimension of $\left(  \overline
{U},\bullet\right)  $ treated as a vector space over the field $\mathbb{Z}%
_{2}$.
\end{theorem}

\begin{proof}
The proof is by induction on $\delta\left(  \mathsf{Con}\left(  \mathbf{A}%
\right)  \right)  $. If $\delta\left(  \mathsf{Con}\left(  \mathbf{A}\right)
\right)  =0$, then the theorem is trivial. Let now $\delta\left(
\mathsf{Con}\left(  \mathbf{A}\right)  \right)  >0$. From $1$-regularity, we
can find $a\in A$ such that $\Theta_{\mathbf{A}}\left(  1,a\right)  $ is an
atom of $\mathsf{Con}\left(  \mathbf{A}\right)  $. Take $\mu\in\mathsf{Cm}%
\left(  \mathbf{A}\right)  $ such that $\Theta_{\mathbf{A}}\left(  1,a\right)
\nleqslant\mu$. As $\mu\wedge\Theta_{\mathbf{A}}\left(  1,a\right)
=\mathbf{0}_{\mathbf{A}}$, we have $\mu^{+}=\mu\vee\Theta_{\mathbf{A}}\left(
1,a\right)  $. Put $V:=\mu/\!\!\sim$. It is easy to observe that,
$\Theta_{\mathbf{A}}\left(  1,a\right)  \leq\varphi$ for every $\varphi
\in\mathsf{Cm}\left(  \mathbf{A}\right)  \backslash V$, since otherwise
$I\left[  \varphi,\varphi^{+}\right]  \searrow I\left[  \mathbf{0}%
_{\mathbf{A}},\Theta_{\mathbf{A}}\left(  1,a\right)  \right]  \nearrow
I\left[  \mu,\mu^{+}\right]  $, a contradiction. Then $M(a):=\left\{  \eta
\in\mathsf{Cm}\left(  \mathbf{A}\right)  :\left(  1,a\right)  \in\eta\right\}
=(M\left(  a\right)  \cap V)\cup(\mathsf{Cm}\left(  \mathbf{A}\right)
\backslash V)$. Clearly, $\mathsf{Con}(\mathbf{A/}\Theta_{\mathbf{A}}\left(
1,a\right)  )$ is isomorphic to the interval $I\left[  \Theta_{\mathbf{A}%
}\left(  1,a\right)  ,\mathbf{1}_{\mathbf{A}}\right]  $ in $\mathsf{Con}%
\left(  \mathbf{A}\right)  $, and so $\mathsf{Cm}(\mathbf{A/}\Theta
_{\mathbf{A}}\left(  1,a\right)  )=\{\eta/\Theta_{\mathbf{A}}\left(
1,a\right)  :\eta\in M(a)\}$. We use Corollary \ref{iff} to deduce that
$\varphi\sim\psi$ in $\mathsf{Con}(\mathbf{A})$ if and only if $\varphi
\mathbf{/}\Theta_{\mathbf{A}}\left(  1,a\right)  \sim\psi/\Theta_{\mathbf{A}%
}\left(  1,a\right)  $ in $\mathsf{Con}(\mathbf{A/}\Theta_{\mathbf{A}}\left(
1,a\right)  )$ for $\varphi,\psi\in M(a)$. Thus, taking the quotient set
$\mathsf{Cm}\left(  \mathbf{A/}\Theta_{\mathbf{A}}\left(  1,a\right)  \right)
/\!\!\sim$ we get, in fact, the same equivalence classes (up to the natural
isomorphism) as in $M(a)/\!\!\sim$. What is more, for any $W\in\mathsf{Cm}%
\left(  \mathbf{A}\right)  /\!\!\sim$, $W\neq V$, we have $W\cap M(a)=W$.

Moreover, in the next section we show (Proposition \ref{hyperplane}) that
$\overline{V\cap M\left(  a\right)  }$ is a hyperplane in $\left(
\overline{V},\bullet\right)  $, and so $\dim\overline{V\cap M(a)}%
+1=\dim\overline{V}$. Combining all these facts, and using induction
assumptions for $\mathbf{A/}\Theta_{\mathbf{A}}\left(  1,a\right)  $, we get:
\begin{align*}
\delta\left(  \mathsf{Con}\left(  \mathbf{A}\right)  \right)   &
=\delta\left(  \mathsf{Con}(\mathbf{A/}\Theta_{\mathbf{A}}\left(  1,a\right)
)\right)  +1\\
&  =\sum\{\dim\overline{U}:U\in\mathsf{Cm}\left(  \mathbf{A/}\Theta
_{\mathbf{A}}\left(  1,a\right)  \right)  /\!\!\sim\}+1\\
&  =\dim\overline{V\cap M(a)}+1+\sum\{\dim\overline{W}:W\in\mathsf{Cm}\left(
\mathbf{A}\right)  /\!\!\sim,W\neq V\}\\
&  =\sum\{\dim\overline{W}:W\in\mathsf{Cm}\left(  \mathbf{A}\right)
/\!\!\sim\}\text{.}%
\end{align*}

\end{proof}

\section{Representations}

Let us start by recalling the well-known construction. Let $\mathcal{P}%
=(P,\leq)$ be a poset. For $S\subset P$ we write $S\!\uparrow\,:=\left\{
\mu\in P:\mu\geq\varphi\text{ for some }\varphi\in S\right\}  $ and
$S\!\downarrow\,:=\left\{  \mu\in P:\mu\leq\varphi\text{ for some }\varphi\in
S\right\}  $. The set $Up\left(  \mathcal{P}\right)  =\left\{  S\subset
P:S=S\uparrow\right\}  $ has the natural structure of Heyting algebra with the
operations: $\cup$, $\cap$, $0:=\emptyset$, $1:=P$, and $\rightarrow$ defined
by $S\rightarrow T:=\left(  \left(  S\backslash T\right)  \downarrow\right)
^{\prime}$ for $S,T\in Up\left(  \mathcal{P}\right)  $. Then the equivalence
operation $\leftrightarrow$ in $Up\left(  \mathcal{P}\right)  $ is given by
\[
S\leftrightarrow T:=(S\rightarrow T)\cap(T\rightarrow S)=\left(  \left(  S\div
T\right)  \downarrow\right)  ^{\prime}%
\]
for $S,T\in Up\left(  \mathcal{P}\right)  $. Alternatively we can also define
$S\leftrightarrow T$ as the largest $C\in Up\left(  \mathcal{P}\right)  $
fulfilling $C\cap S=C\cap T$.

In this section we apply the above construction to the poset $\mathsf{Cm}%
\left(  \mathbf{A}\right)  $. Now, let $\mathbf{A}$ be an arbitrary algebra.
Then, it is well known that the map
\[
M:\mathsf{Con}\left(  \mathbf{A}\right)  \ni\varphi\rightarrow M\left(
\varphi\right)  :=\left\{  \mu\in\mathsf{Cm}\left(  \mathbf{A}\right)
:\varphi\leq\mu\right\}  \in Up\left(  \mathsf{Cm}\left(  \mathbf{A}\right)
\right)
\]
is one-to-one, as $\varphi=\bigwedge M\left(  \varphi\right)  $ holds for
every $\varphi\in\mathsf{Con}\left(  \mathbf{A}\right)  $ (Birkhoff's theorem).

In the next proposition that shows how the principal congruences behave under
the map $M$, we merely assume that $\mathrm{H}(\mathbf{A})$ is congruence orderable.

\begin{proposition}
\label{eqv}Let $\mathrm{H}(\mathbf{A})$ be congruence orderable (with respect
to a constant term $1$). Then
\[
M(\Theta_{\mathbf{A}}(a,b))=M(a)\leftrightarrow M(b)\text{,}%
\]
for $a,b\in A$, where $M(c):=M(\Theta_{\mathbf{A}}(1,c))=\{\mu\in
\mathsf{Cm}\left(  \mathbf{A}\right)  :(1,c)\in\mu\}$ for $c\in A$.
\end{proposition}

\begin{proof}
Let $a,b\in A$. It is enough to show that $M\left(  \Theta_{\mathbf{A}%
}(a,b)\right)  $ is the largest $C\in Up\left(  \mathsf{Cm}\left(
\mathbf{A}\right)  \right)  $ such that $C\cap M\left(  a\right)  =C\cap
M\left(  b\right)  $. Clearly, $M\left(  \Theta_{\mathbf{A}}(a,b)\right)  \cap
M\left(  a\right)  =M\left(  \Theta_{\mathbf{A}}(a,b)\right)  \cap M\left(
b\right)  $. Let $C\in Up\left(  \mathsf{Cm}\left(  \mathbf{A}\right)
\right)  $ fulfill $C\cap M\left(  a\right)  =C\cap M\left(  b\right)  $.
Assume that there exists $\mu\in C\backslash M\left(  \Theta_{\mathbf{A}%
}(a,b)\right)  $. Then $(a,b)\notin\mu$. From the fact that $\mathrm{H}%
(\mathbf{A})$ is congruence orderable we get $\mu\vee\Theta_{\mathbf{A}%
}\left(  1,a\right)  \neq\mu\vee\Theta_{\mathbf{A}}\left(  1,b\right)  $.
Without loss of generality we can assume that $\mu\vee\Theta_{\mathbf{A}%
}\left(  1,a\right)  \nleq\mu\vee\Theta_{\mathbf{A}}\left(  1,b\right)  $.
Then there exists $\upsilon\in\mathsf{Cm}\left(  \mathbf{A}\right)  $ such
that $\mu\vee\Theta_{\mathbf{A}}\left(  1,a\right)  \leq\upsilon$ and $\left(
1,b\right)  \notin\upsilon$. Hence $\upsilon\in C\cap M\left(  a\right)  $ and
$\upsilon\notin M\left(  b\right)  $, a contradiction.
\end{proof}

From now on, as in the preceding section, we assume that $\mathbf{A}$ is
strongly Fregean. Let us consider the following subfamily of $Up\left(
\mathsf{Cm}\left(  \mathbf{A}\right)  \right)  $, which for finite algebras
can be identified with $\mathsf{Con}\left(  \mathbf{A}\right)  $.

\begin{definition}
$\mathcal{S}\left(  \mathbf{A}\right)  :=\{S\in Up\left(  \mathsf{Cm}\left(
\mathbf{A}\right)  \right)  :\overline{S\cap U}$ is a subalgebra of $\left(
\overline{U},\bullet\right)  $ for every $U\in\mathsf{Cm}\left(
\mathbf{A}\right)  /\!\!\sim\}$.
\end{definition}

From Proposition \ref{baslem} it is easy to deduce the following lemma.

\begin{lemma}
\label{herlem}Let $\mathbf{A}$ be strongly Fregean, $S$ be a finite set from
$\mathcal{S}\left(  \mathbf{A}\right)  $ and $\mu\in\mathsf{Cm}\left(
\mathbf{A}\right)  $ be such that $\bigwedge S\leq\mu$. Then $\mu\in S$.
\end{lemma}

\begin{proof}
Let $S=\left\{  \alpha_{1},\ldots,\alpha_{n}\right\}  \subset\mathsf{Cm}%
\left(  \mathbf{A}\right)  $ and $\alpha_{1}\wedge\cdots\wedge\alpha_{n}%
\leq\mu$. According to Proposition \ref{baslem} there exist $\mu_{1}%
,\ldots,\mu_{n}\in\mu/\!\!\sim\cup\left\{  \mathbf{1}_{\mathbf{A}}\right\}  $
such that $\alpha_{i}\leq\mu_{i}$ for $i=1,\ldots,n$, $\mu_{1}\wedge
\cdots\wedge\mu_{n}\leq\mu$, and $\mu=\mu_{i_{1}}\bullet\cdots\bullet
\mu_{i_{k}}$ for some $1\leq i_{1}<\ldots<i_{k}\leq n$. Put $U=\mu/\!\!\sim$.
From $S=S\uparrow$ we get $\left\{  \mu_{1},\ldots,\mu_{n}\right\}
\backslash\left\{  \mathbf{1}_{\mathbf{A}}\right\}  \subset S\cap U$. Since
$\overline{S\cap U}$ is a subgroup of $\left(  \overline{U},\bullet\right)  $
we deduce that $\mu\in\overline{S\cap U}$.\ Hence $\mu\in S$.
\end{proof}

As a consequence, we get

\begin{theorem}
\label{S(A)}Let $\mathbf{A}$ be strongly Fregean. Then

\begin{enumerate}
\item $M(\mathsf{Con}\left(  \mathbf{A}\right)  )\subset\mathcal{S}\left(
\mathbf{A}\right)  $.

\item If $\mathbf{A}$ is finite, then $M$ establishes a one-to-one
correspondence between $\mathsf{Con}\left(  \mathbf{A}\right)  $ and
$\mathcal{S}\left(  \mathbf{A}\right)  $.
\end{enumerate}
\end{theorem}

\begin{proof}
It follows immediately from the definition of the operation `$\bullet$' that
the image of $M$ is contained in $\mathcal{S}\left(  \mathbf{A}\right)  $. To
prove (2) assume that $S\in\mathcal{S}\left(  \mathbf{A}\right)  $. Clearly,
$S\subset M\left(  \bigwedge S\right)  $. From Lemma \ref{herlem} we deduce
that $S=M\left(  \bigwedge S\right)  $, which completes the proof.
\end{proof}

Our ultimate aim here is to characterize, for $\mathbf{A}$ with a Malcev term,
those sets from $Up\left(  \mathsf{Cm}\left(  \mathbf{A}\right)  \right)  $
that correspond to congruences $\left\{  \Theta_{\mathbf{A}}\left(
1,a\right)  :a\in A\right\}  $, or in other words, due to $1$-regularity of
$\mathbf{A}$, to elements of $A$. We start from the definition of the family
of hereditary sets in $Up\left(  \mathsf{Cm}\left(  \mathbf{A}\right)
\right)  $ representing the principal congruences.

Let $\eta\in\mathsf{Cm}\left(  \mathbf{A}\right)  $ and let $U=\eta/$%
\noindent$\sim$. We put $U^{+}:=U\!\!\uparrow\!\!\backslash\hspace{0.02in}U$.
Then%
\begin{align*}
U^{+}  &  =\left\{  \vartheta\in\mathsf{Cm}\left(  \mathbf{A}\right)
:\vartheta>\eta\right\} \\
&  =\left\{  \vartheta\in\mathsf{Cm}\left(  \mathbf{A}\right)  :\vartheta
>\varphi\text{ for all }\varphi\in U\right\} \\
&  =\left\{  \vartheta\in\mathsf{Cm}\left(  \mathbf{A}\right)  :\vartheta
\geq\eta^{+}\right\} \\
&  =M(\eta^{+})\text{.}%
\end{align*}

\begin{definition}
Let $Z\subset\mathsf{Cm}\left(  \mathbf{A}\right)  $. We say that $Z$ is
\textsl{hereditary}, if:

(I) $Z=Z\!\!\uparrow$;

(II) for all $U\in\mathsf{Cm}\left(  \mathbf{A}\right)  /\!\!\sim$\ , if
$U^{+}\subset Z$, then $\overline{Z\cap U}=\overline{U}$ or $\overline{Z\cap
U}$ is \linebreak a hyperplane in $\left(  \overline{U},\bullet\right)  $.

(We use the word `hyperplane' because we can interpret a Boolean group as a
vector space over the field $\mathbb{Z}_{2}$.)

We denote the set of all hereditary subsets of $\mathsf{Cm}\left(
\mathbf{A}\right)  $ by $\mathcal{H}\left(  \mathbf{A}\right)  $. Clearly,
$\mathcal{H}\left(  \mathbf{A}\right)  \subset\mathcal{S}\left(
\mathbf{A}\right)  $.
\end{definition}

\begin{remark}
Observe that if $U\in\mathsf{Cm}\left(  \mathbf{A}\right)  /\!\!\sim$,
$Z\in\mathcal{H}\left(  \mathbf{A}\right)  $\ and $U^{+}\nsubseteq Z$, then
$Z\cap U=\emptyset$. Converse implication is true, if $\left|  U\right|  >1$.
Moreover, if $U=\eta/$\noindent$\sim$, $\eta\in\mathsf{Cm}\left(
\mathbf{A}\right)  $, and $U^{+}\subset Z$, then $\bigwedge Z\leq\eta^{+}$.
\end{remark}

\begin{remark}
It follows from Proposition \ref{distributive} that if we assume additionally
that $\mathsf{Con}\left(  \mathbf{A}\right)  $ is distributive, then
$\mathcal{H}\left(  \mathbf{A}\right)  =\mathcal{S}\left(  \mathbf{A}\right)
=Up\left(  \mathsf{Cm}\left(  \mathbf{A}\right)  \right)  $.
\end{remark}

\begin{proposition}
\label{hyperplane}Let $\mathbf{A}$ be strongly Fregean and let $a,b\in A$.
Then $M(\Theta_{\mathbf{A}}(a,b))\in\mathcal{H}\left(  \mathbf{A}\right)  $.
\end{proposition}

\begin{proof}
Let $a,b\in A$. From Theorem \ref{S(A)} we deduce that, $M(\Theta_{\mathbf{A}%
}(a,b))\in\mathcal{S}\left(  \mathbf{A}\right)  $. Let $U=\eta/\!\!\sim$ for
some $\eta\in\mathsf{Cm}\left(  \mathbf{A}\right)  $. Assume that
$U^{+}\subset M(\Theta_{\mathbf{A}}(a,b))$. Clearly, $U^{+}=M(\eta^{+})$. Then
$\Theta_{\mathbf{A}}(a,b)=\bigwedge M(\Theta_{\mathbf{A}}(a,b))\leq\bigwedge
U^{+}=\eta^{+}$. If $\left(  a,b\right)  \in0_{U}$, then $M\left(
\Theta_{\mathbf{A}}(a,b)\right)  \cap U=U$. On the other hand, if $\left(
a,b\right)  \notin0_{U}$, then $M\left(  \Theta_{\mathbf{A}}(a,b)\right)  \cap
U\varsubsetneq U$. For $\mu,\varphi\in U\backslash M(\Theta_{\mathbf{A}%
}(a,b))$, we have $\mu\bullet\varphi\in\overline{M(\Theta_{\mathbf{A}%
}(a,b))\cap U}$. In consequence, $\overline{M\left(  \Theta_{\mathbf{A}%
}(a,b)\right)  \cap U}$ is a hyperplane in $\left(  \overline{U}%
,\bullet\right)  $, as desired.
\end{proof}

It turns out that $\mathcal{H}\left(  \mathbf{A}\right)  $ inherits a natural
equivalence operation from the Heyting algebra $(Up\left(  \mathsf{Cm}\left(
\mathbf{A}\right)  \right)  ,\cup,\cap,\emptyset,\mathsf{Cm}\left(
\mathbf{A}\right)  ,\rightarrow)$.

\begin{theorem}
\label{H(A)eq}Let $\mathbf{A}$ be strongly Fregean. Then $\mathcal{H}\left(
\mathbf{A}\right)  $ is closed under the equivalence operation
$\leftrightarrow$ in $Up\left(  \mathsf{Cm}\left(  \mathbf{A}\right)  \right)
$.
\end{theorem}

\begin{proof}
Let $S,T\in\mathcal{H}(\mathbf{A})$. Then $S\leftrightarrow T=\left(  \left(
S\div T\right)  \downarrow\right)  ^{\prime}\in Up\left(  \mathsf{Cm}\left(
\mathbf{A}\right)  \right)  $. To prove that $S\leftrightarrow T$ is
hereditary it is enough to show for all $U\in\mathsf{Cm}\left(  \mathbf{A}%
\right)  /\!\!\sim$\ , if $U^{+}\subset S\leftrightarrow T$, then
$\overline{(S\leftrightarrow T)\cap U}=\overline{U}$ or $\overline
{(S\leftrightarrow T)\cap U}$ is a hyperplane in $\left(  \overline{U}%
,\bullet\right)  $. Let $U\in\mathsf{Cm}\left(  \mathbf{A}\right)  /\!\!\sim$
and $U^{+}\subset S\leftrightarrow T$. We know that $U^{+}\cap S=U^{+}\cap T$.
Let us consider two possibilities: 1) $U^{+}\nsubseteq S$ and $U^{+}\nsubseteq
T$ (and then $U\cap S=U\cap T=\emptyset$); 2) $U^{+}\subset S\cap T$.

1) In this situation we show that $U\subset S\leftrightarrow T$. Assume that
$\mu\in U\backslash(S\leftrightarrow T)$. Then $\mu\in\left(  S\div T\right)
\downarrow$, and so there is $\varphi\in\mathsf{Cm}\left(  \mathbf{A}\right)
$ such that $\mu\leq\varphi\in S\div T$. There is no loss of generality in
assuming that $\varphi\in S\backslash T$. Clearly $\mu<\varphi$, since
otherwise $\mu\in U\cap S=\emptyset$, a contradiction. Hence $\varphi\in
U^{+}\cap S$ and $\varphi\notin U^{+}\cap T$, which contradicts our assumption.

2) In this case $\overline{S\cap U}$ either equals $\overline{U}$ or is a
hyperplane in $\left(  \overline{U},\bullet\right)  $. The same is true for
$\overline{T\cap U}$. It is straightforward to prove that for $\mu\neq\nu$ and
$\mu,\nu\in\left(  S\div T\right)  ^{\prime}\cap U=(S\cap T\cap U)\cup
(S^{\prime}\cap T^{\prime}\cap U)$ or $\mu,\nu\in\left(  S\div T\right)  \cap
U$ we have $\mu\bullet\nu\in\left(  S\div T\right)  ^{\prime}\cap U$. Hence
$\overline{\left(  S\div T\right)  ^{\prime}\cap U}$ is either $\overline{U}$
or a hyperplane in $\left(  \overline{U},\bullet\right)  $. Now, observe that
$S\div T\subset\left(  S\div T\right)  \downarrow$. Then $(S\leftrightarrow
T)\cap U\subset\left(  S\div T\right)  ^{\prime}\cap U$. It is enough to show
that these sets are equal. Assume that $\mu\in\left(  S\div T\right)
^{\prime}\cap U$ and $\mu\notin S\leftrightarrow T$. Then $\mu\in\left(  S\div
T\right)  \downarrow$. Thus, there is $\varphi\in\mathsf{Cm}\left(
\mathbf{A}\right)  $ such that $\mu\leq\varphi\in S\div T$. Clearly, $\mu
\neq\varphi$, and so $\mu<\varphi$. In consequence, $\varphi\in U^{+}\subset
S\cap T$, a contradiction.
\end{proof}

In \cite[Theorem 3.8]{IdzSloWro09} the following result, showing that the
equivalential algebras form a paradigm of congruence permutable Fregean
varieties, was proven. Recall that equivalential algebras are defined as
equivalential subreducts of Heyting algebras \cite{KabWro75}. For more
information on equivalential algebras, see \cite{Slo08}.

\begin{theorem}
\label{Malcev}Let $\mathrm{H}(\mathbf{A})$ be congruence orderable with
respect to a constant term$~1$. Then the following conditions are equivalent:

\begin{enumerate}
\item $\mathbf{A}$ has a Malcev term;

\item $\mathbf{A}$ has a binary term $e$, satisfying one of two equivalent conditions:

\begin{enumerate}
\item $e$ is a principle congruence term, i.e., $\Theta_{A}\left(  a,b\right)
=\Theta_{A}\left(  1,e\left(  a,b\right)  \right)  $ for all $a,b\in A$;

\item $e$-reduct of $\mathbf{A}$ is an equivalential algebra.
\end{enumerate}
\end{enumerate}

Note that if these conditions are satisfied, then $\mathbf{A}$ is strongly Fregean.
\end{theorem}

Now, from Proposition \ref{eqv} and Theorem \ref{Malcev} we get

\begin{corollary}
\label{e}Let $\mathbf{A}$ be strongly Fregean with a principle congruence term
$e$. Then $M$ preserves equivalence operation, that is
\[
M(e(a,b))=M(a)\leftrightarrow M(b)
\]
for $a,b\in A$.
\end{corollary}

In particular, the assumptions of Theorem \ref{Malcev} are true if
$\mathbf{A}$ is finite, congruence orderable and congruence permutable.
Namely, the following proposition holds:

\begin{proposition}
\label{H(A)}Let $\mathbf{A}$ be a finite, congruence orderable (with respect
to a constant term$~1$) and congruence permutable algebra. Then $\mathrm{H}%
(\mathbf{A})$ is congruence orderable.
\end{proposition}

\begin{proof}
Let $\alpha\in\mathsf{Con}\left(  \mathbf{A}\right)  $, $a,b\in A$ fulfill
$\alpha\vee\Theta_{\mathbf{A}}\left(  1,a\right)  =\alpha\vee\Theta
_{\mathbf{A}}\left(  1,b\right)  $. It is enough to show that $(a,b)\in\alpha
$. On the contrary, suppose that $(a,b)\notin\alpha$. There is no loss of
generality in assuming that $a\nleqslant b$. Then, from congruence
permutability of $\mathbf{A}$, there exists $c\in A$ such that
$b\overset{\alpha}{\equiv}c\overset{\Theta_{\mathbf{A}}\left(  1,a\right)
}{\equiv}1$. Thus, $c>a$ and $\left(  b,c\right)  \in\alpha$. Let $c_{0}\in A$
be a maximal element with this property. Hence, $\left(  1,b\right)  \in
\alpha\vee\Theta_{\mathbf{A}}\left(  1,c_{0}\right)  $, and so $\left(
1,a\right)  \in\alpha\vee\Theta_{\mathbf{A}}\left(  1,c_{0}\right)  $.
Consequently, there is $d\in A$ such that $a\overset{\alpha}{\equiv
}d\overset{\Theta_{\mathbf{A}}\left(  1,c_{0}\right)  }{\equiv}1$. Clearly, we
have $d\geq c_{0}>a$ and $\left(  1,b\right)  \in\alpha\vee\Theta_{\mathbf{A}%
}\left(  1,d\right)  $, and so there exists $t\in A$ such that
$b\overset{\alpha}{\equiv}t\overset{\Theta_{\mathbf{A}}\left(  1,d\right)
}{\equiv}1$. Hence, $t\geq d\geq c_{0}>a$ and $(b,t)\in\alpha$. Finally, from
maximality of $c_{0}$ we get $t=c_{0}=d$, and so $a\overset{\alpha}{\equiv
}d=c_{0}\overset{\alpha}{\equiv}b$, a contradiction.
\end{proof}

\begin{remark}
In fact, the congruence permutability condition in the above proposition can
be weakened. We see from the proof that it suffices to assume that
$\mathbf{A}$ is \textsl{congruence }$1$\textsl{-permutable}, i.e.,
$(1,a)\in\alpha\vee\beta$ iff $(1,a)\in\alpha\circ\beta$ for $\alpha,\beta
\in\mathsf{Con}\left(  \mathbf{A}\right)  $, $a\in A$.
\end{remark}

The following theorem extends the result proven for equivalential algebras in
\cite[Theorem 5]{Slo05}.

\begin{theorem}
\label{main}Let $\mathbf{A}$ is finite, congruence orderable (with respect to
a constant term$~1$) with a Malcev term. Then

\begin{enumerate}
\item $\mathcal{V}(\mathbf{A)}$ is Fregean with a principle congruence term
$e$ that turns every algebra in $\mathcal{V}(\mathbf{A)}$ into an
equivalential algebra;

\item the map%
\[
A\ni a\rightarrow M\left(  a\right)  :=\left\{  \mu\in\mathsf{Cm}\left(
\mathbf{A}\right)  :\left(  1,a\right)  \in\mu\right\}  \in\mathcal{H}\left(
\mathbf{A}\right)
\]
establishes a one-to-one correspondence between $A$ and $\mathcal{H}\left(
\mathbf{A}\right)  $. Moreover, \linebreak$\left(  \mathcal{H}\left(
\mathbf{A}\right)  ,\leftrightarrow\right)  $ is an equivalential algebra
isomorphic with $(A,e)$.
\end{enumerate}
\end{theorem}

\begin{proof}
(1) From \cite[Theorem 2.10]{IdzSloWro09} to show that $\mathcal{V}%
(\mathbf{A})$ is Fregean for finite $\mathbf{A}$, it is enough to prove that
$\mathrm{HS}(\mathbf{A})$ is congruence orderable and $\mathcal{V}%
(\mathbf{A})$ is $1$-regular. Let $\mathbf{B}\in\mathrm{S}(\mathbf{A})$ and
$\Theta_{\mathbf{B}}\left(  1,a\right)  =\Theta_{\mathbf{B}}\left(
1,b\right)  $ for $a,b\in B$. Clearly, $\Theta_{\mathbf{B}}\left(  1,a\right)
\leq$ $\Theta_{\mathbf{A}}\left(  1,a\right)  \cap B^{2}$, and so $\left(
1,b\right)  \in\Theta_{\mathbf{A}}\left(  1,a\right)  $. Consequently,
$\Theta_{\mathbf{A}}\left(  1,b\right)  \leq\Theta_{\mathbf{A}}\left(
1,a\right)  $, and, analogously, $\Theta_{\mathbf{A}}\left(  1,a\right)
\leq\Theta_{\mathbf{A}}\left(  1,b\right)  $. From congruence orderability of
$\mathbf{A}$ we get $a=b$. Thus $\mathrm{S}(\mathbf{A})$ is congruence
orderable. By Proposition \ref{H(A)}, $\mathrm{HS}(\mathbf{A})$ is congruence
orderable. To show that $\mathcal{V}(\mathbf{A})$ is $1$-regular, we use
Theorem \ref{Malcev}. Since $\mathrm{H}(\mathbf{A})$ is congruence orderable
and $\mathbf{A}$ has a Malcev term, we deduce that there exists a binary term
$e$ such that $e$ is a principle congruence term in $\mathbf{A}$, i.e.,
$\Theta_{A}\left(  a,b\right)  =\Theta_{A}\left(  1,e\left(  a,b\right)
\right)  $ for all $a,b\in A$, and $(A,e)$ belongs to the variety of
equivalential algebras \cite{KabWro75}. Hence, $(C,e)$ must be an
equivalential algebra for every $\mathbf{C}\in\mathcal{V}(\mathbf{A})$. Let
$\alpha\in\mathsf{Con}\left(  \mathbf{C}\right)  \subset\mathsf{Con}\left(
(C,e)\right)  $. Thus $\left(  c,d\right)  \in\alpha$ iff $(1,e(c,d))\in
\alpha$. In consequence, $\mathcal{V}(\mathbf{A})$ is $1$-regular with a
principle congruence term $e$.

(2) By (1) $\mathbf{A}$ is strongly Fregean with a principle congruence term.
Now, from Proposition \ref{hyperplane} we know that the image of $M$ is
contained in $\mathcal{H}\left(  \mathbf{A}\right)  $. On the other hand, the
injectivity of $M$ follows from Theorem \ref{S(A)}.2. Moreover, from Corollary
\ref{e} we deduce that $M$ preserved equivalence operation. Hence it is enough
to prove that $M$ is surjective.

Let $Z\in\mathcal{H}\left(  \mathbf{A}\right)  $. Assume that $a$ is a minimal
(with respect to the natural order `$\leq$' in $A$) element in $\left\{  c\in
A:\left(  1,c\right)  \in\bigwedge Z\right\}  $. We show that $Z=M\left(
a\right)  $. Clearly, $Z\subset M\left(  a\right)  $. Suppose that
$Z\varsubsetneq M\left(  a\right)  $. Let $\mu$ be a maximal element in
$\left\{  \eta\in\mathsf{Cm}\left(  \mathbf{A}\right)  :\eta\in M\left(
a\right)  \backslash Z\right\}  $. Put $U=\mu/\!\!\sim$. Then $Z\cap U\neq U$
and $U^{+}=\left\{  \vartheta\in\mathsf{Cm}\left(  \mathbf{A}\right)
:\vartheta>\mu\right\}  \subset Z$. Hence, as $Z\in\mathcal{H}\left(
\mathbf{A}\right)  $, we know that $\overline{Z\cap U}$ is a hyperplane in
$\left(  \overline{U},\bullet\right)  $. Moreover, $Z\cap U\varsubsetneq
M\left(  a\right)  \cap U$. Then $\overline{M\left(  a\right)  \cap U}$ is not
a hyperplane, and so $M\left(  a\right)  \cap U=U$. Thus $\left(  1,a\right)
\in0_{U}$. It follows from Lemma \ref{herlem} that $\mu\notin Z$ implies
$\bigwedge Z\nleq\mu$. Since $\bigwedge Z\leq\bigwedge U^{+}=\mu^{+}$, we get
$I\left[  \mu,\mu^{+}\right]  \searrow I\left[  \bigwedge Z\wedge\mu,\bigwedge
Z\right]  $ and consequently $\operatorname*{card}I\left[  \bigwedge
Z\wedge\mu,\bigwedge Z\right]  =2$. Take a join irreducible $\beta
\in\mathsf{Con}\left(  \mathbf{A}\right)  $ such that $\beta\leq\bigwedge Z$
and $\beta\nleq\bigwedge Z\wedge\mu$. Let $\beta^{-}$ denote the unique
subcover of $\beta$ in $\mathsf{Con}\left(  \mathbf{A}\right)  $. Then
$I\left[  \bigwedge Z\wedge\mu,\bigwedge Z\right]  \searrow I[\bigwedge
Z\wedge\mu\wedge\beta,\beta]=I[\beta^{-},\beta]$. Consequently, $I\left[
\mu,\mu^{+}\right]  \searrow I\left[  \beta^{-},\beta\right]  $. From the fact
that $\mathbf{A}$ is $1$-regular and $\beta$ is join irreducible we deduce
that there exists $b\in A$ such that $\beta=\Theta_{\mathbf{A}}\left(
1,b\right)  $, and so $(1,b)\notin\mu$ and $(1,b)\in\bigwedge Z$. In
consequence, $(a,b)\notin\mu$ and $(a,b)\in\bigwedge Z$. From the fact that
$e$ is a binary principle congruence term in $\mathbf{A}$ we get
$\Theta_{\mathbf{A}}\left(  a,b\right)  =\Theta_{\mathbf{A}}\left(  1,e\left(
a,b\right)  \right)  $. Thus $\left(  1,e\left(  a,b\right)  \right)
\notin\mu$ and $\left(  1,e\left(  a,b\right)  \right)  \in\bigwedge Z$.
Hence, $M(a)\neq M\left(  e\left(  a,b\right)  \right)  $, and so from the
minimality of $a$ in $\left\{  c\in A:\left(  1,c\right)  \in\bigwedge
Z\right\}  $, we get $M\left(  e\left(  a,b\right)  \right)  \nsubseteq
M\left(  a\right)  $. Let $\gamma$ be a maximal element in $\left\{
\varphi\in\mathsf{Cm}\left(  \mathbf{A}\right)  :\varphi\in M\left(  e\left(
a,b\right)  \right)  \backslash M\left(  a\right)  \right\}  $ and
$W:=\gamma/\!\!\sim$. Then for $\upsilon\in\mathsf{Cm}\left(  \mathbf{A}%
\right)  $ we know that $\upsilon>\gamma$ implies $\upsilon\in M\left(
a\right)  $, which gives $\left(  1,a\right)  \in\gamma^{+}$, and, since
$\left(  1,e\left(  a,b\right)  \right)  \in\gamma^{+}$, we get $\left(
1,b\right)  \in\gamma^{+}$. Note that $\left(  1,a\right)  \notin0_{W}$ as
$\left(  1,a\right)  \notin\gamma$, and so $W\neq U$. Let us consider two
cases: $\left(  1,b\right)  \in0_{W}$ and $\left(  1,b\right)  \notin0_{W}$.
If $\left(  1,b\right)  \in0_{W}$, then $\left(  a,b\right)  ,\left(
1,b\right)  \in\gamma$, and we obtain $\left(  1,a\right)  \in\gamma$, a
contradiction. If $\left(  1,b\right)  \notin0_{W}$, then $\left(  1,b\right)
\in\gamma^{+}\backslash0_{W}$, and so there exists $\psi\in\mathsf{Con}\left(
\mathbf{A}\right)  $ such that $\left(  1,b\right)  \notin\psi\in W$. Then
$I\left[  \beta^{-},\beta\right]  \nearrow I\left[  \psi,\psi^{+}\right]  $
which gives $\mu\sim\psi$, and consequently, $U=W$, a contradiction.
\end{proof}

The assumption that $\mathbf{A}$ has a Malcev term cannot be dropped. If
$\mathbf{A}$ does not come from congruence permutable variety, then we can
represent congruences of the form $\Theta_{\mathbf{A}}(1,a)$ as elements of
$\mathcal{H}\left(  \mathbf{A}\right)  $, however, the latter set can be
larger than the image of $M$. To show this it is enough to consider the
following simple example from \cite{PrzSlo17}.

\begin{example}
Take $\mathbf{A}=(A,\rightarrow)$, where $A:=\{a,b,1\}$, and $\rightarrow$ is
given by $x\rightarrow y:=1$, where $x=y$, and $x\rightarrow y:=y$ where
$x\neq y$. In fact, $\mathbf{A}$ is a Hilbert algebra as the implicative
subreduct of the Boolean algebra $\mathbf{B}:=\mathbf{2}\times\mathbf{2}$,
where $\mathbf{2}$ denotes the two-element Boolean algebra, and so
$\mathcal{V}(\mathbf{A})$ is Fregean and congruence distributive. Then
$\mathsf{Cm}\left(  \mathbf{A}\right)  $ is consisted of two incomparable
congruences $\alpha:=\Theta_{\mathbf{A}}\left(  1,a\right)  $ and
$\beta:=\Theta_{\mathbf{A}}\left(  1,b\right)  $. Hence $\mathcal{H}\left(
\mathbf{A}\right)  =Up\left(  \mathsf{Cm}\left(  \mathbf{A}\right)  \right)
=\{\emptyset,\{\alpha\},\{\beta\},\{\alpha,\beta\}\}$ has four elements,
whereas $\left\vert A\right\vert =3$.
\end{example}

\end{document}